\documentclass[11pt, a4paper]{amsart}
\usepackage{amssymb,amsthm, dsfont, a4}
\usepackage{hyperref, xcolor}
\usepackage{verbatim}   
\usepackage{enumitem}
\usepackage[all]{xy}
\usepackage[active]{srcltx}
\usepackage[short,nodayofweek]{datetime}

\definecolor{my-blue}{cmyk}{1,0.6,0,0}

\definecolor{my-green}{cmyk}{0.8,0,1,0.5}

\hypersetup{
colorlinks=true, 
linkcolor=my-blue, 
citecolor=my-green,
urlcolor=black
}

\newcommand\FF{{\mathbb F}}

\newcommand\NN{{\mathbb N}}

\newcommand{\Fp}{\FF_{\!p}}

\newcommand\fH{{\mathfrak H}}
\newcommand\fT{{\mathfrak T}}

\newcommand\m{{\mathfrak m}}

\newcommand{\G}{\mathcal{G}}
\renewcommand{\H}{\mathcal{H}}

\newcommand{\id}{\mathrm{id}}

\newcommand{\trdeg}{\mathrm{trdeg}}

\newcommand{\ps}[1]{[\![#1]\!]}  

\newcommand{\cat}[1]{\mathsf{#1}}

\newcommand{\isom}{\cong}

\renewcommand{\th}[1]{\theta^{(#1)}}

\DeclareMathOperator{\Hom}{Hom}

\DeclareMathOperator{\Aut}{Aut}

\DeclareMathOperator{\uGal}{\underline{Gal}}

\DeclareMathOperator{\Quot}{Quot}
\DeclareMathOperator{\Spec}{Spec}
\def\markdef{\bf }

\theoremstyle{plain}
\newtheorem{thm}{Theorem}[section]

\newtheorem{lem}[thm]{Lemma}
\newtheorem{prop}[thm]{Proposition}

\newtheorem*{mainthm}{Main Theorem}

\theoremstyle{definition}
\newtheorem{defn}[thm]{Definition}
\newtheorem{exmp}[thm]{Example}
\newtheorem{rem}[thm]{Remark}

\usepackage{datetime}
\newdate{date}{17}{05}{2019}

\begin{document}

\title[Reduced group schemes as ID-Galois groups]{Reduced group schemes as iterative differential Galois groups}
\author{Andreas Maurischat}
\address{\rm {\bf Andreas Maurischat}, Lehrstuhl A f\"ur Mathematik, RWTH Aachen University, Germany }
\email{\sf andreas.maurischat@matha.rwth-aachen.de}

\date{\displaydate{date}}



\begin{abstract}
This article is on the inverse Galois problem in Galois theory of linear iterative differential equations in positive
characteristic. We show that it has an affirmative answer for reduced algebraic group schemes over any iterative differential field which is finitely generated over its algebraically closed field of constants.

We also introduce the notion of equivalence of iterative derivations on a given field - a condition which implies that the inverse Galois problem over equivalent iterative derivations are equivalent.
\end{abstract}

\maketitle

\setcounter{tocdepth}{1}
\tableofcontents

\section{Introduction}

In differential Galois theory in positive characteristic, classical derivations are replaced by iterative derivations in order to keep the constants ``small'', and to obtain an appropriate Galois theory. For example in characteristic zero, the rational function field $C(t)$ with the derivation $\partial_t=\frac{d}{dt}$ has as constants $C=\{ f\in C(t)\mid \partial_t(f)=0\}$. In positive characteristic $p$, however, the constants would be the subfield $C(t^p)$.
Therefore, one considers the iterative derivation $\theta_t$ instead which is given by a collection of $C$-linear maps $(\th{n}_t)_{n\geq 0}$ determined by
\[   \th{n}_t(t^k) =\binom{k}{n} t^{k-n} \]
for all $k,n\in \NN$. Morally, $\th{n}_t$ equals $\frac{1}{n!}\partial_t^n$ although, the latter expression is not meaningful in characteristic $p<n$.
Then the constants $\{ f\in C(t) \mid \forall n>0: \th{n}_t(f)=0\}$ equal $C$, as one is used to from characteristic zero.

Using these iterative derivations, a Galois theory for linear iterative differential equations (usually called \textit{Picard-Vessiot theory}) has been developed by Matzat and van der Put \cite{bhm-mvdp:ideac} quite analogous to the Galois theory for linear differential equations in characteristic zero. 
This theory was improved by the author in \cite{am:gticngg} (see also \cite{am:igsidgg}) to a Galois theory whose Galois correspondence also takes into account intermediate fields over which the solution field is inseparable.
It even happens that the solution field itself is inseparable over the base field, in which case the Galois group is a non-reduced affine algebraic group scheme.

The inverse Galois problem asks which groups or group schemes, respectively, arise as Galois groups for given iterative differential fields.

In this article, we consider the inverse Galois problem over base differential fields which are finitely generated over their field of constants. Our main theorem is the following.

\begin{mainthm} (Theorem \ref{thm:realisation})
Let $(L,\theta)$ be an  iterative differential field of positive characteristic $p$
 with algebraically closed field of constants $C$ such that $L$ is finitely generated over $C$, and $L\neq C$. Then every reduced affine algebraic group scheme over $C$ is the Galois group of some iterative differential module over $(L,\theta)$.
\end{mainthm}

Here, we use ``algebraic'' synonymous to ``of finite type''.

The proof of this theorem is an adaptation of the proof of the corresponding statement in characteristic zero \cite[Cor.~4.13]{ab-dh-jh:dgglsf}, and it uses that every  reduced affine algebraic group scheme over $C$ occurs as iterative differential Galois group over the standard iterative differential field $(C(t),\theta_t)$ which is given in \cite[Thm.~8.8]{bhm-mvdp:cdgt}.

However, there is a major point which does not transfer from the proof in characteristic zero. One has to show that one can replace the given iterative derivation $\theta$ on such an ID-field $L$ by another one such that 
\begin{enumerate}
\item the inverse Galois problem does not change, and
\item the new ID-field contains  the standard ID-field $(C(t),\theta_t)$ for an appropriate element $t\in L$.
\end{enumerate}
This is the content of Section~\ref{sec:equivalent-IDs}.

\section{Notation}

In this section, we recall the main setting of Picard-Vessiot theory for iterative differential equations, as far as we will need it. For more details, we refer to the articles mentioned above (\cite{am:gticngg}, \cite{am:igsidgg}).

All rings are assumed to be commutative with unit. The set $\NN$ of natural numbers contains $0$.

An {\markdef iterative derivation} $\theta$ on a ring $R$ is a ring homomorphism
$\theta:R\to R\ps{T}$ into the power series ring over $R$ with the properties:
\begin{enumerate}[label={(\roman*)}]
\item $\theta^{(0)}(r)=r$ for all $r\in R$,
\item \label{item:iterativity} $\theta^{(i)}\circ \theta^{(j)}=\binom{i+j}{i}\theta^{(i+j)}$, for all $i,j\geq
0$,
\end{enumerate}
where the maps $\theta^{(i)}:R\to R$ are defined by
$\theta(r)=:\sum_{i=0}^\infty \theta^{(i)}(r)T^i$. The {\markdef ring of constants} of $R$ is denoted by
\[ R^\theta:=\{ r\in R\mid \theta(r)=r\cdot T^0\} = \{ r\in R\mid \forall i\geq 1:\,\th{i}(r)=0 \}. \]
The pair $(R,\theta)$ is called an {\markdef iterative differential ring}, or {\markdef ID-ring} for short.

If there is need to emphasize the extra
variable $T$ or if we use another name for the variable, we add a subscript to $\theta$, i.e.~denote
the iterative derivation by $\theta_T$ (resp. $\theta_U$ if the variable is named $U$).

For a given homomorphism of rings $f:R\to S$, we denote by $f\ps{T}$ the $T$-linear extension of $f$ to a homomorphism $R\ps{T}\to S\ps{T}$ of the power series rings. To be precise, we should speak of the continuous $T$-linear extension with respect to the $T$-adic topology on $R\ps{T}$ and $S\ps{T}$.
Actually, on all power series rings (also in several variables), we will use the $\m$-adic topology where $\m$ is the ideal generated by the variables, and all homomorphisms will be continuous with respect to these topologies.

\medskip

With this notation, the iterativity condition \ref{item:iterativity} for $\theta$ is equivalent to the commutativity of the diagram

\centerline{\xymatrix@C+10pt{
R \ar[r]^{\theta_U} \ar[d]_{\theta_T} & R\ps{U}
\ar[d]^{U\mapsto U+T} \\
R\ps{T} \ar[r]^(.45){\theta_U\ps{T}} & R\ps{U,T},
}}
\noindent or in other terms $\theta_U\ps{T}\circ \theta_T=\theta_{U+T}$.

\medskip

For an ID-ring $(R,\theta)$, an {\markdef iterative differential module} (or {\markdef ID-module} for short) $(M,\theta_M)$  over $(R,\theta)$ consists of a finitely generated $R$-module $M$ and an additive map $\theta_M:M\to M\ps{T}$ such that 
\begin{enumerate}[label={(\roman*)}] 
\item[(o)] $\theta_M(rm)=\theta(r)\theta_M(m)$, for all $r\in R$, $m\in M$,
\item $\theta_M^{(0)}=\id_M$, and
\item $\theta_M^{(i)}\circ \theta_M^{(j)}=\binom{i+j}{i}\theta_M^{(i+j)}$ for all $i,j\geq 0$,
\end{enumerate}
where the maps $\theta_M^{(i)}:M\to M$ are defined by $\theta_M(m)=:\sum_{i=0}^\infty \theta_M^{(i)}(m)T^i$.
Elements $m\in M$ such that $\theta_M(m)=m\in M\ps{T}$ are called {\markdef constant elements}.

If we use another variable or expression instead of $T$ or would like to emphasize $T$, we add another subscript to $\theta_M$, e.g.~$ \theta_{M,U}:M\to M\ps{U}, m\mapsto \sum_{i=0}^\infty \theta_M^{(i)}(m)U^i,$ and the iteration rule can be rewritten as
$\theta_{M,U}\ps{T}\circ \theta_{M,T}=\theta_{M,U+T}$.

An {\markdef ID-homomorphism} $\varphi:(M,\theta_M)\to (N,\theta_N)$ between two ID-modules $(M,\theta_M)$ and $(N,\theta_N)$ over $(R,\theta)$ is an $R$-module homomorphism $\varphi:M\to N$ such that $\varphi\ps{T}\circ \theta_M=\theta_N\circ \varphi$.

There should be no confusion between the iterative derivation $\theta_M$ on a module $M$, and the iterative derivation on the ring with emphasis on the used variable, e.g.~$\theta_U$, as letters $M$ and $N$ will always denote modules, and the names of variables are restricted to $T$ and $U$ and expressions thereof.

\begin{defn}
An iterative derivation $\theta$ on a ring $R$ is called {\markdef trivial}, if $\theta(r)=rT^0$ for all $r\in R$, or in other words, if $R^\theta=R$. It is called {\markdef non-trivial} if it is not trivial. 
\end{defn}

We gather a well-known fact for iterative derivations in the following lemma.

\begin{lem}\label{lem:existence-of-d}
Let $\theta$ be a non-trivial iterative derivation on a field $F$ of characteristic $p>0$.
\begin{enumerate}
\item For any non-constant $f\in F$, there is a unique integer $d\geq 0$ (depending on $f$) such that
$\th{p^d}(f)\neq 0$, and $\th{m}(f)=0$ for all $m$ which are not divisible by $p^d$. 
\item There is a unique  integer $d\geq 0$ such that
$\th{p^d}\neq 0$, and $\th{m}= 0$ for all $m$ which are not divisible by $p^d$. 
\end{enumerate}
\end{lem}

\begin{proof}
Of course, an integer $d$ satisfying the two given properties in (1) or in (2), resp., has to be unique, since it must be the smallest $d$ such that $\th{p^d}(f)\neq 0$, or $\th{p^d}\neq 0$, respectively.
Furthermore, for the second part, one just has to take the minimal $d$ that occurs in the first part.
For obtaining the existence of $d$ for a given non-constant $f\in F$, we use the iterativity condition \ref{item:iterativity} of $\theta$. Namely, for a given positive integer $m$, write $m=m_0+m_1p+\cdots + m_rp^r$ in base $p$ expansion with $m_0,\ldots, m_r\in \{0,\ldots, p-1\}$. Then the iteration rule implies that
\[   (\th{p^r})^{m_r}\circ (\th{p^{r-1}})^{m_{r-1}} \circ \ldots \circ (\th{p^0})^{m_0}=c\cdot \th{m} \]
for a non-zero element $c\in \Fp$ (see \cite[Sect.~2]{bhm-mvdp:ideac}).
Hence, if $d\geq 0$ is the smallest integer such that $\th{p^d}(f)\neq 0$, then for every $m\geq 1$ which is not divisible by $p^d$, one of the elements $m_0,\ldots, m_{d-1}$ is non-zero, and hence 
$\th{m}(f)=0$.
\end{proof}

We will also be concerned with algebraic extensions.

\begin{lem}\label{lem:algebraic-extension}
Let $(F,\theta)$ be an ID-field of characteristic $p$, and let $L$ be an algebraic extension of $F$.
\begin{enumerate}
\item If $L/F$ is separable, then the iterative derivation $\theta$ uniquely extends to an iterative derivation on $L$.
\item If $L/F$ is not separable, there is at most one extension of $\theta$ to $L$.
\end{enumerate}
\end{lem}

\begin{proof}
The first part is due to F.~K.~Schmidt, and a proof is given in \cite[Sect.~2.1]{bhm-mvdp:ideac}.
For the second part, assume that there exists an extension of $\theta$ to $L$.
Since for any $l\in L$ some power $l^{p^j}$ is separable over $F$, this extension has to agree on $F(l^{p^j})$ with
the unique extension of $\theta$ to $F(l^{p^j})$ stated in the first part. Hence, $\theta(l)$ has to equal the uniquely determined element $\left( \theta(l^{p^j})\right)^{-p^j}\in L\ps{T}$.
\end{proof}

\subsection*{Picard-Vessiot theory}

We now fix an ID-field $(F,\theta)$, and let $C$ be its field of constants.

A {\markdef Picard-Vessiot ring} for an ID-module $(M,\theta_M)$ is an ID-ring extension $(R,\theta)$ of $(F,\theta)$ satisfying
\begin{enumerate} 
\item $R^\theta=F^\theta$, i.e.~$R$ has the same constants as $F$,
\item $R$ is ID-simple, i.e.~$R$ has no non-trivial ideals which are stable under all $\th{k}$,
\item $R$ is minimal with the property that $R\otimes_F M$ has a basis of constant elements.
\end{enumerate}
As in the differential setting in characteristic zero, such a Picard-Vessiot ring is an integral domain, and its field of fractions $E$ is a {\markdef Picard-Vessiot field} for $M$, i.e.~a minimal ID-field extension of $F$, having the same constants, and such that $E\otimes_F M$ has a basis of constant elements.

For a Picard-Vessiot ring $R$ with corresponding Picard-Vessiot field $E$, the {\markdef Galois group} is defined to be the  functor
$$\uGal(E/F):=\uGal(R/F): \cat{Alg}_C \to \cat{Groups},
D\mapsto \Aut^{\theta}(R\otimes_C D/F\otimes_C D)$$
where $\Aut^{\theta}$ denotes the group of ring automorphisms commuting with all $\th{n}$, and 
the $C$-algebra $D$ is equipped with the trivial iterative derivation, i.e.~$R\otimes_C D$ is an extension by constants.
This group functor turns out to be representable by $(R\otimes_F R)^\theta$, and the latter algebra is indeed finitely generated over $C$. Therefore, the group functor equals the affine algebraic group scheme $\Spec\bigl( (R\otimes_F R)^\theta\bigr)$ over $C$.
Moreover, $\Spec(R)$ is a $\uGal(R/F)$-torsor over $F$, and the torsor property is induced by the natural left-$R$-linear homomorphism of ID-rings
\begin{equation}\label{eq:torsor-isom}
  R\otimes_C (R\otimes_F R)^\theta \to R\otimes_F R
\end{equation}
which indeed is an isomorphism.
This also implies that $\trdeg(E/F)=\dim(\uGal(E/F))$.

Although, there is a full Galois correspondence between intermediate ID-fields $F\subseteq L\subseteq E$ and closed subgroup schemes $\H\subseteq \G$, we will only need the restricted Galois correspondence given in \cite{am:pvtdsr} which works on the ring level, and does not need the detour to the Picard-Vessiot fields.

\begin{thm}\label{thm:galois-correspondence} (\cite[Thm.~6.12]{am:pvtdsr})
Let $R$ be a PV-ring for some ID-module over $F$, and $\G=\uGal(R/F)$. Then there is a bijection between
$$\fT:=\{ T \mid F\subseteq T\subseteq R \text{ intermediate PV-ring} \}$$
and
$$\fH:=\{ \H \mid \H\leq \G \text{ closed normal subgroup scheme of }\G \}$$
given by
$\Psi:\fT\to \fH, T\mapsto \uGal(R/T)$ resp.~$\Phi:\fH\to \fT, \H\mapsto R^{\H}$.
\end{thm}
Here, the invariants $R^\H$ are defined to be
\[  R^\H:=\{ r\in R\mid \forall D\in  \cat{Alg}_C, h\in \H(D): h(r\otimes 1)=r\otimes 1 \}. \]
Of course, the Galois group of $R^\H$ over $F$ is the factor group $\G/\H$.

Although it is not stated in the theorem, the Galois correspondence is even an anti-isomorphisms of poset lattices, i.e.~joins are mapped to intersections and vice versa. For later use, we explicitly mention a special case with its proof.

\begin{lem}\label{lem:pv-ring-is-tensor-product}
Let $R$ be a PV-ring over $F$ whose Galois group is a product $\G_1\times \G_2$ of two  group schemes $\G_1,\G_2$ over $C$. Let $R_1:=R^{\G_1\times 1}$ and $R_2:=R^{1\times \G_2}$. Then $R\isom R_1\otimes_F R_2$.
\end{lem}

\begin{proof}
Consider the homomorphism of ID-rings $\varphi:R_1\otimes_F R_2\to R$ given by the embeddings. Our aim is to show that it is an isomorphism.

By the Galois correspondence, $R_1$ and $R_2$ are Picard-Vessiot rings over $F$ with $\uGal(R_1/F)=(\G_1\times \G_2)/(\G_1\times 1)\isom \G_2$ and $\uGal(R_2/F)=(\G_1\times \G_2)/(1\times \G_2)\isom \G_1$.
The isomorphism \eqref{eq:torsor-isom} then shows that $R_1\otimes_C C[\G_2]\isom R_1\otimes_F R_1$ and
$R_2\otimes_C C[\G_1]\isom R_2\otimes_F R_2$.\\
Therefore, using both isomorphisms after tensoring with $R$,
\begin{eqnarray*}
R\otimes_F (R_1\otimes_F R_2) &\isom & (R\otimes_F R_1)\otimes_R (R\otimes_F R_2) 
\isom  (R \otimes_C C[\G_2])\otimes_R (R\otimes_C C[\G_1]) \\
&\isom & R \otimes_C (C[\G_2] \otimes_C C[\G_1]) 
\isom  R \otimes_C C[\G_1\times \G_2]
\end{eqnarray*}
On the other hand, applying the isomorphism \eqref{eq:torsor-isom} to $R$ directly leads to
\[ R\otimes_F R \isom R \otimes_C C[\G_1\times \G_2] \]
Hence, $R\otimes_F (R_1\otimes_F R_2) \stackrel{\isom}{\longrightarrow} R\otimes_F R$.

A closer look on the isomorphism reveals that this isomorphism is indeed $\id_R\otimes \varphi$, and since $R/F$ is faithfully flat, $\varphi$ is an isomorphism.
\end{proof}

\section{Equivalence of iterative derivations}\label{sec:equivalent-IDs}

In this section, we introduce equivalence of iterative derivations which allows us to change the iterative derivation on an ID-field without changing the inverse Galois problem.

\begin{defn}\label{def:equivalent-IDs}
Let $F$ be a field, and let $\theta$ and $\tilde{\theta}$ be two iterative derivations on $F$. We call
$\theta$ and $\tilde{\theta}$ {\markdef equivalent}, if there exists a  continuous isomorphism of $F$-algebras
$\lambda:F\ps{T}\to F\ps{T}$ such that $\tilde{\theta}=\lambda\circ \theta$.
\end{defn}

\begin{rem}
Of course, any continuous homomorphism of $F$-algebras $\lambda:F\ps{T}\to F\ps{T}$ is uniquely determined by
the image $\lambda(T)=\sum_{n=0}^\infty a_nT^n\in F\ps{T}$ of $T$. To be well-defined we need $\lambda(T)\in TF\ps{T}$, i.e.~$a_0=0$. The homomorphism $\lambda$ is an isomorphism, if in addition $a_1\neq 0$.

We denote the image of $T$ under such a homomorphism $\lambda$ by $P_\lambda(T)\in F\ps{T}$ for being able to also consider $P_\lambda(U)\in F\ps{U}$ for another indeterminate $U$.

The iterative derivation $\tilde{\theta}$ can then be written in terms of $\theta$ and $P_\lambda(T)$ as
\[  \tilde{\theta}(f)=\sum_{k=0}^\infty \th{k}(f)P_\lambda(T)^k \]
for all $f\in F$.
\end{rem}

\begin{lem}\label{lem:cond-on-lambda}
Let $F$ be a field, and let $\theta$ be a non-trivial iterative derivation  on $F$. Further, let
$\lambda:F\ps{T}\to F\ps{T}$ be a  continuous homomorphism of $F$-algebras. We denote $P_\lambda(T):=\lambda(T)\in F\ps{T}$, and let
$\tilde{\theta}:=\lambda\circ \theta:F\to F\ps{T}$. The following are equivalent:
\begin{enumerate}
\item The map $\tilde{\theta}$ is an iterative derivation.
\item $ P_\lambda(U+T)= P_\lambda(U)+ \tilde{\theta}_U\ps{T}(P_\lambda(T))\in F\ps{U,T}$.
\end{enumerate}
\end{lem}

\begin{proof}
As $\theta$ is an iterative derivation, and $\lambda$ is a homomorphism of $F$-algebras, $\tilde{\theta}$ is a ring homomorphism and satisfies $\tilde{\theta}^{(0)}=\id_F$. Hence, $\tilde{\theta}$ is an iterative derivation if and only if the iteration rule $\tilde{\theta}_U\ps{T}\circ \tilde{\theta}_T
= \tilde{\theta}_{U+T}$ holds.

We compute both sides in terms of $\theta$ and $P_\lambda$. We have for $f\in F$:
\begin{eqnarray*}
\left(\tilde{\theta}_U\ps{T}\circ \tilde{\theta}_T\right)(f) &=& \tilde{\theta}_U\ps{T}\left( \sum_{k=0}^\infty \theta^{(k)}(f)P_\lambda(T)^k \right)\\
&=& \sum_{k=0}^\infty \tilde{\theta}_U\bigl( \theta^{(k)}(f) \bigr) \left( \tilde{\theta}_U\ps{T}(P_\lambda(T))\right)^k \\
&=&  \sum_{k,n=0}^\infty \theta^{(n)}(\theta^{(k)}(f)) P_\lambda(U)^n \left( \tilde{\theta}_U\ps{T}(P_\lambda(T))\right)^k \\
&=&  \sum_{k,n=0}^\infty \binom{n+k}{n} \theta^{(n+k)}(f) P_\lambda(U)^n \left( \tilde{\theta}_U\ps{T}(P_\lambda(T))\right)^k \\
&=&  \sum_{m=0}^\infty  \theta^{(m)}(f) \sum_{n=0}^m \binom{m}{n} P_\lambda(U)^n \left( \tilde{\theta}_U\ps{T}(P_\lambda(T))\right)^{m-n} \\
&=& \sum_{m=0}^\infty  \theta^{(m)}(f) \left( P_\lambda(U) + \tilde{\theta}_U\ps{T}(P_\lambda(T))\right)^m,
\end{eqnarray*}
where we used iterativity of $\theta$ in the forth equality. The term $ \tilde{\theta}_{U+T}(f)$ is given as
\begin{eqnarray*}
 \tilde{\theta}_{U+T}(f) &=& \sum_{m=0}^\infty  \theta^{(m)}(f) P_\lambda(U+T)^m.
\end{eqnarray*}
Hence, $ \tilde{\theta}$ is iterative if and only if for all $f\in F$,
\begin{eqnarray*}
0 &=& \sum_{m=1}^\infty  \theta^{(m)}(f) \left( P_\lambda(U+T)^m - Q(U,T)^m \right) 
\end{eqnarray*}
where $Q(U,T):=P_\lambda(U) + \tilde{\theta}_U\ps{T}(P_\lambda(T))$.

If $f\in F$ is not constant, let $d\geq 0$ such that $\th{p^d}(f)\neq 0$, and that $\th{m}(f)=0$ for all $m$ which are not divisible by $p^d$, as in Lemma \ref{lem:existence-of-d}. Then
\begin{eqnarray*}
0 &=&  \sum_{m=1}^\infty  \theta^{(m)}(f) \left( P_\lambda(U+T)^m - Q(U,T)^m \right) \\ &=&
  \sum_{j=1}^\infty  \theta^{(jp^d)}(f) \left( P_\lambda(U+T)^{jp^d} - Q(U,T)^{jp^d} \right) \\
  &=&  \sum_{j=1}^\infty  \theta^{(jp^d)}(f) \left( P_\lambda(U+T)^{j} - Q(U,T)^{j} \right)^{p^d} \\
  &=& \bigl( P_\lambda(U+T) - Q(U,T) \bigr)^{p^d} \\ & & \qquad \cdot 
  \left( \theta^{(p^d)}(f)+  \sum_{j=2}^\infty  \theta^{(jp^d)}(f) \left(  \sum_{i=0}^{j-1}
  P_\lambda(U+T)^{j-1-i} Q(U,T)^{i}  \right)^{\!\!p^d}  \right).
\end{eqnarray*}
As $P_\lambda(U+T)$ and $Q(U,T)$ are both in the maximal ideal $(U,T)F\ps{U,T}$, the right factor in the last line is invertible in $F\ps{U,T}$. Hence, the iteration rule holds at a non-constant element $f\in F$, if and only if $P_\lambda(U+T) - Q(U,T)=0$.

This shows the equivalence of the two conditions.
\end{proof}

\begin{exmp}
We consider an arbitrary non-trivial iterative derivation $\theta$ on $F$, and 
the homomorphism $\lambda:F\ps{T}\to F\ps{T}$ given by $\lambda(T)=T^{p^d}$ for some $d\geq 1$, as well as  $\tilde{\theta}:=\lambda\circ \theta$.  Then $P_\lambda(T)= T^{p^d}$, and 
$ \tilde{\theta}_U\ps{T}(P_\lambda(T))=\tilde{\theta}_U\ps{T}(T^{p^d})=T^{p^d}$ independent of $\theta$. Hence,
\[  P_\lambda(U+T)=(U+T)^{p^d} =U^{p^d}+T^{p^d}=P_\lambda(U)+ \tilde{\theta}_U\ps{T}(P_\lambda(T)). \]
Therefore, by the previous lemma, the map $ \tilde{\theta}$  is an iterative derivation.

Be aware that here $\theta$ and $\tilde{\theta}$ are not equivalent, as $\lambda$ is not an isomorphism.
\end{exmp}

\begin{prop}\label{prop:lambda-unique}
Let $F$ be a field and $\theta$ and $\tilde{\theta}$ be two equivalent non-trivial iterative derivations on $F$. Then the isomorphism $\lambda$ in Definition \ref{def:equivalent-IDs} is unique.
\end{prop}

\begin{proof}
Let $f\in F$ be a non-constant element, and let $d\geq 0$ such that $\th{p^d}(f)\neq 0$, and that $\th{m}(f)=0$ for all $m$ which are not divisible by $p^d$, as in Lemma \ref{lem:existence-of-d}. 
 Then from the equality
\[  \sum_{n=1}^\infty \tilde{\theta}^{(n)}(f)T^n =  \sum_{k=1}^\infty \theta^{(k)}(f)P_\lambda(T)^{k}
= \th{p^d}(f)P_\lambda(T)^{p^d} + \sum_{k=2}^\infty \theta^{(jp^d)}(f)P_\lambda(T)^{jp^d}, \]
 one obtains
 \[ P_\lambda(T)^{p^d} = \th{p^d}(f)^{-1}\cdot \left(\sum_{n=1}^\infty \tilde{\theta}^{(n)}(f)T^n
 -  \sum_{j=2}^\infty \theta^{(jp^d)}(f)P_\lambda(T)^{jp^d} \right). \]
This is a recursive formula for the coefficients of $P_\lambda(T)^{p^d}$, and hence also for the coefficients of $P_\lambda(T)$. Hence, $P_\lambda(T)$ and $\lambda$ are unique.
\end{proof}

\begin{thm}\label{thm:equivalence-of-categories}
Let $F$ be a field and $\theta$ and $\tilde{\theta}$ be two equivalent non-trivial iterative derivations on $F$, as well as 
$\lambda:F\ps{T}\to F\ps{T}$ such that $\tilde{\theta}=\lambda\circ \theta$, and $P_\lambda(T):=\lambda(T)\in F\ps{T}$.
\begin{enumerate}
\item \label{item:change-of-ID}
For an ID-module $(M,\theta_M)$ over $(F,\theta)$, the $F$-vector space $M$ together with 
the map 
\begin{eqnarray*} \tilde{\theta}_M:=(\id_M\otimes \lambda)\circ \theta_M:M &\longrightarrow & M\ps{T}=M\otimes_F F\ps{T} \\  m &\longmapsto &  \sum_{k=0}^\infty \theta_M^{(k)}(m)P_\lambda(T)^k
\end{eqnarray*}
is an ID-module $(M,\tilde{\theta}_M)$ over $(F,\tilde{\theta})$.
\item The assignment $(M,\theta_M) \mapsto (M,\tilde{\theta}_M)$ of part \eqref{item:change-of-ID} induces
an equivalence of categories between  the category of ID-modules over $(F,\theta)$ and the  category of ID-modules over $(F,\tilde{\theta})$.
\end{enumerate}
\end{thm}

\begin{proof}
For the first part, we have to show that $\tilde{\theta}_M$ is an iterative derivation over $\tilde{\theta}$.
Clearly,
\[  \tilde{\theta}_M^{(0)}(m)=\theta_M^{(0)}(m)=m \]
for all $m\in M$. Furthermore, for $f\in F$, $m\in M$:
\begin{eqnarray*}
\tilde{\theta}_M(fm) &=& (\id_M\otimes \lambda)\left( \theta(f)\cdot \theta_M(m)\right) \\
&=& (\lambda\circ \theta)(f)\cdot ( (\id_M\otimes \lambda)\circ \theta_M)(m) \\
&=& \tilde{\theta}(f)\cdot \tilde{\theta}_M(m).
\end{eqnarray*}
Finally, we have to show that $\tilde{\theta}_{M,U}\ps{T}\circ \tilde{\theta}_{M,T}$ equals $\tilde{\theta}_{M,U+T}$. By the same computation as in the proof of Lemma~\ref{lem:cond-on-lambda},
using that $\theta_M$ is iterative, we obtain for $m\in M$:
\[ \left( \tilde{\theta}_{M,U}\ps{T}\circ \tilde{\theta}_{M,T}\right)(m) = 
\sum_{k=0}^\infty  \theta_M^{(k)}(m) \left( P_\lambda(U) + \tilde{\theta}_U\ps{T}(P_\lambda(T))\right)^k, \]
as well as
\[  \tilde{\theta}_{M,U+T}(m) = \sum_{k=0}^\infty  \theta_M^{(k)}(m) P_\lambda(U+T)^k. \]
As $\tilde{\theta}$ is an iterative derivation, these are equal by Lemma~\ref{lem:cond-on-lambda}.

For the second part, we first recognize that the assignment is a bijection between ID-modules over $(F,\theta)$ and ID-modules over $(F,\tilde{\theta})$, as one has an inverse assignment by changing the roles of $\theta$ and $\tilde{\theta}$. We have to check that a homomorphism $\varphi\in \Hom_F(M,N)$ which is an ID-homomorphism $(M,\theta_M)\to  (N,\theta_N)$ also is an ID-homomorphism 
$(M,\tilde{\theta}_M)\to (N,\tilde{\theta}_N)$. If $\varphi$ is an ID-homomorphism from $(M,\theta_M)$ to $(N,\theta_N)$, the left square of the diagram below commutes. The right square of this diagram commutes, since both compositions equal $\varphi\otimes \lambda$. Hence, the whole diagram commutes which means that $\varphi$ is also an ID-homomorphism between $(M,\tilde{\theta}_M)$ and $(N,\tilde{\theta}_N)$.
\[
\centerline{\xymatrix@+10pt{
M \ar[r]^(.4){\theta_M} \ar[d]_{\varphi}& M\ps{T} \ar[r]^{\id_M\otimes \lambda} \ar[d]_{\varphi\ps{T}} 
\ar@{-->}[dr]^{\varphi\otimes \lambda}
& M\ps{T} \ar[d]^{\varphi\ps{T}}  \\
N \ar[r]^(.4){\theta_N} & N\ps{T} \ar[r]^{\id_N\otimes \lambda} & N\ps{T}   
}}
\]

As the assignment is clearly functorial, this finishes the proof.
\end{proof}

In order to handle the inverse Galois problem, the previous proposition allows us to replace a given iterative derivation by an equivalent one, as the categories of iterative differential modules above them are equivalent.

The next proposition provides us with a ``nice'' iterative derivation in those equivalence classes for which $\th{1}\ne 0$.

\begin{prop}\label{prop:extension-of-id-by-t}
Let $(F,\theta)$ be an iterative differential field with field of constants $C$. Assume that $\th{1}\neq 0$, and let $t\in F$ with $\th{1}(t)\neq 0$.
Further, we define a continuous isomorphism of $F$-algebras $\mu:F\ps{T}\to F\ps{T}$ by $\mu(T)=\theta(t)-t$, and let $\lambda:=\mu^{-1}$.

Then the homomorphism $\tilde{\theta}:=\lambda\circ \theta:F\to F\ps{T}$ is an iterative derivation on $F$, and 
$\tilde{\theta}$ extends the iterative derivation by $t$ on $C(t)$, i.e.~$\tilde{\theta}(t)=t+T$.
\end{prop}

\begin{proof}
The second claim is immediate from the definitions. Namely,
\[ \mu\left( \tilde{\theta}(t)\right) = \theta(t)=t+\mu(T)=\mu(t+T), \]
and hence, $\tilde{\theta}(t)=t+T$, because $\mu$ is an isomorphism.

It remains to show that $\tilde{\theta}$ is an iterative derivation. As in Lemma \ref{lem:cond-on-lambda}, the properties $\tilde{\theta}^{(0)}=\id_F$  and that $\tilde{\theta}$ is a ring homomorphism are clear.
The condition for iterativity given there, however, is not so easy to check. So we proceed differently.

We define the continuous homomorphism of $F$-algebras
\begin{eqnarray*}
\nu:F\ps{U,T}&\longrightarrow & F\ps{U,T} \quad \text{by} \\
U &\longmapsto & P_\mu(U)=\theta_U(t)-t\\
 T &\longmapsto & \theta_U\ps{T}(P_\mu(T)).
\end{eqnarray*}
The map $\nu$  actually is an isomorphism by the formal inverse function theorem (cf.~\cite[Theorem 1.1.2]{avde:pajc}), since $\nu(U)\equiv \theta^{(1)}(t)\cdot U \mod{(U,T)^2}$ and
$\nu(T)\equiv \theta^{(1)}(t)\cdot T \mod{(U,T)^2}$.

We then have to show that the big square in the following diagram commutes:

\centerline{\xymatrix@C+15pt{
F \ar[rrr]^{\tilde{\theta}_U} \ar@{=}[rd] \ar[ddd]_{\tilde{\theta}_T}& \ar @{} [dr] |(.4)*+[o][F-]{1} & & F\ps{U} \ar[dl]^{\mu_U} \ar[ddd]^{U\mapsto U+T} \\
\ar @{} [dr] |(.4)*+[o][F-]{2} & F  \ar[r]^{\theta_U} \ar[d]_{\theta_T} \ar @{} [dr] |(.4)*+[o][F-]{3} & F\ps{U} \ar[d]^(.55){U\mapsto U+T} \ar @{} [dr] ^(.4)*+[o][F-]{4} &  \\
 & F\ps{T} \ar[r]^(.45){\theta_U\ps{T}} \ar @{} [dr] |(.4)*+[o][F-]{5} &  F\ps{U,T} & \\
 F\ps{T} \ar[ur]^{\mu} \ar[rrr]^{\tilde{\theta}_U\ps{T}} & &  & F\ps{U,T} \ar[ul]^{\nu}
}
}
\noindent where $\mu_U:F\ps{U}\to F\ps{U},U\mapsto P_\mu(U)$ is the map $\mu$ with $T$ replaced by $U$ both in the source and the domain.

As $\nu$ is an isomorphism, it suffices to prove that all the smaller quadrilaterals commute.

The quadrilaterals $1$ and $2$ commute by definition of $\tilde{\theta}$, and quadrilateral $3$ commutes, since $\theta$ is an iterative derivation. The maps building quadrilateral $4$ are all continuous homomorphisms of $F$-algebras, and therefore we just have to check that both compositions agree on $U$:
\[  \mu_U(U)|_{U\mapsto U+T}=(\theta_U(t)-t)|_{U\mapsto U+T}=\theta_{U+T}(t)-t \]
and
\begin{eqnarray*}
\nu(U+T) &=& P_\mu(U)+  \theta_U\ps{T}(P_\mu(T))=\theta_U(t)-t + \theta_U\ps{T}(\theta_T(t)-t) \\
&=& \theta_U(t)-t + \theta_{U+T}(t)- \theta_U(t) =\theta_{U+T}(t)-t,
\end{eqnarray*}
where we  used $\theta_U\ps{T}(\theta_T(t))=\theta_{U+T}(t)$ by iterativity of $\theta$. So quadrilateral $4$ also commutes, and we are left with quadrilateral $5$. As both compositions are continuous ring homomorphisms, it suffices to check commutativity for $T$ and for elements $f\in F$.

By definition of $\nu$, we have
\[  \nu(\tilde{\theta}_U\ps{T}(T))=\nu(T)=\theta_U\ps{T}(P_\mu(T))=(\theta_U\ps{T}\circ \mu)(T). \]
Finally, as $\nu|_{F\ps{U}}=\mu_U$, and $\mu|_F=\id_F$, 
\[ \nu(\tilde{\theta}_U\ps{T}(f))=\nu(\tilde{\theta}_U(f)) 
=(\mu_U\circ \tilde{\theta}_U)(f)=\theta_U(f)=\theta_U\ps{T}(\mu(f)). \qedhere \]
\end{proof}

\section{Inverse Galois problem for reduced group schemes}\label{sec:inverse-galois-problem}

In this section, the main result is the following

\begin{thm}\label{thm:realisation}
Let $(L,\theta)$ be an  iterative differential field of positive characteristic $p$
 with algebraically closed field of constants $C$ such that $L$ is finitely generated over $C$, and $L\neq C$. Then every reduced affine algebraic group scheme over $C$ is the Galois group of some iterative differential module over $(L,\theta)$.\end{thm}

The theorem is also true in characteristic zero and is proven in \cite[Thm.~4.12]{ab-dh-jh:dgglsf}. Even more, our proof will follow the lines of their proof.

The first part is a strategy which is called \textit{Kovacic trick} there. For proving the Kovacic trick in the iterative differential setting, we need a lemma corresponding to \cite[Lemma 4.9]{ab-dh-jh:dgglsf}.

\begin{lem}\label{lem:embedding}
Let $S/R$ be an extension of ID-rings with $R$ being ID-simple. Then $S^\theta\otimes_{R^\theta} R$ embeds into $S$.
\end{lem}


\begin{proof}
We consider the canonical ID-homomorphism $\iota:S^\theta\otimes_{R^\theta} R\longrightarrow S$ given by the embeddings. Its kernel is an ID-ideal of $S^\theta\otimes_{R^\theta} R$. By \cite[Lemma 10.7]{am:gticngg}, the ID-ideals of $S^\theta\otimes_{R^\theta} R$ are in one-to-one correspondence with the ideals of $S^\theta$.\footnote{Finite generation of $S^\theta$ over $R^\theta$ which is assumed there is not needed.} Since, $S^\theta$ is a field this means that $S^\theta\otimes_{R^\theta} R$ is ID-simple and therefore the homomorphism $\iota$ is injective, since it is not zero.
\end{proof}

The Kovacic trick is given in the following proposition.

\begin{prop}\label{prop:kovacic-trick}
Let $(F,\theta)$ be an iterative differential field with field of constants $C:=F^\theta$. Let
$L/F$ be an iterative differential field extension that is finitely generated over $F$, and has the same field of constants $L^\theta=C$. Let $\G$ be an affine algebraic group scheme defined over $C$ with the property that for every $r\in \NN$, $\G^r$ is an iterative differential Galois group over $F$. Then $\G$ is an iterative differential Galois group over $L$.
\end{prop}

\begin{proof}
The proof is almost the same as the proof of \cite[Thm.~4.12]{ab-dh-jh:dgglsf}.
We follow their arguments with a slight modification by using the restricted Galois correspondence of Thm.~\ref{thm:galois-correspondence}.

Let $F^{\rm alg}$ denote the maximal algebraic extension of $F$ to which the iterative derivation $\theta$ can be extended
(comp.~Lemma \ref{lem:algebraic-extension}). For an ID-extension $S$ of $F$, we consider the subextension
\[ S':=\{ x\in S \mid x \text{ is algebraic over }F \}. \]
 As there is only one extension of the iterative derivation on $F$ to $F^{\rm alg}$,
every embedding of $S'$ into  $F^{\rm alg}$ is an ID-embedding, and we will also write $S\cap F^{\rm alg}$ instead of $S'$ with an implicit embedding. Also be aware that if $S$ is finitely generated over $F$, then $S'$ is a finite extension of $F$ (see e.g.~\cite[Ch.~5, \S 15, Cor.~1, p.~117]{nb:aiic4-7}).

Now, let $m:=\mathrm{trdeg}(L/F)+1$ and $d:=[L'':F]$ where $L''$ is the normal closure of $L\cap F^{\rm alg}$ in $F^{\rm alg}$ which of course is independent of the chosen embedding of $L'$ into $F^{\rm alg}$.

By assumption on the group scheme $\G$, there is a Picard-Vessiot ring $R$ over $F$ with
$\uGal(R/F)=\G^{2d+m}$. Taking $\H_i\subset \G^{2d+m}$ (for $i=1,\ldots, 2d+m$) as the normal subgroup where the $i$-th factor is omitted, the invariant subring $R_i:=R^{\H_i}$ is a Picard-Vessiot
ring over $F$ with Galois group $\uGal(R_i/F)=\G^{2d+m}/\H_i\cong \G$ (see Thm.~\ref{thm:galois-correspondence}). Furthermore by Lemma \ref{lem:pv-ring-is-tensor-product}, one has
$R=R_1\otimes_F R_2\otimes_F \ldots\otimes_F R_{2d+m}$.

The idea of the proof is to show that there is at least one index $i$ for which $L\otimes_F R_i$ is a Picard-Vessiot ring over $L$ whose Galois group  still is $\G$. Roughly speaking that for this index $i$, the rings $R_i$ and $L$ are ``differentially independent'' over $F$.

In the first step, one removes tensor factors $R_i$ until the remaining tensor product intersect $L''$ in $F$,
after fixing an embedding of  $R'=R'_1\otimes_F R'_2\otimes_F \ldots\otimes_F R'_{2d+m}$ into $F^{\rm alg}$. By \cite[Lemma~4.11]{ab-dh-jh:dgglsf} (which is purely algebraic), one has to remove at most $d$ factors (actually even only $d-1$ factors). So after relabelling, we are left with $\tilde{R}:=R_1\otimes_F R_2\otimes_F \ldots\otimes_F R_{d+m}$, and $\tilde{R}$ is linearly disjoint to $L''$ over $F$.
If the constants of $\tilde{R}\otimes_F L''$ are strictly larger than $C$, we have to remove further tensor factors $R_i$.
By Lemma \ref{lem:embedding}, $(\tilde{R}\otimes_F L'')^\theta\otimes_C \tilde{R}$ embeds into $\tilde{R}\otimes_F L''$, and hence 
\[ \big[ (\tilde{R}\otimes_F L'')^\theta : C\big]=\big[ (\tilde{R}\otimes_F L'')^\theta\otimes_C \tilde{R} :\tilde{R}\big]
\leq \big[\tilde{R}\otimes_F L'' : \tilde{R}\big]=[L'':F]=d. \]
Therefore, as above, the number of factors to remove is at most $d$. Again after relabelling, we are left with $R_1\otimes_F R_2\otimes_F \ldots\otimes_F R_{m}$.

\medskip

In the second step, we rename $R:=R_1\otimes_F R_2\otimes_F \ldots\otimes_F R_{m}$ with Galois group $\uGal(R/F)=\G^{m}$. One then verifies that by the choice of indices, $R$ is linearly disjoint to $L$ over $F$, and hence $R\otimes_F L$ is an integral domain. Further, $C$ is algebraically closed in $\tilde{E}:=\Quot(R\otimes_F L)$.

\medskip

In the third step, one shows that there is at least one index $i$ such that the field of fractions $\tilde{E}_i:=\Quot(R_i\otimes_F L)$ has constants $C$. Then $\tilde{E}_i$ is a Picard-Vessiot extension of $L$ with Picard-Vessiot ring
$R_i\otimes_F L$ and Galois group $\uGal(R_i\otimes_F L/L)=\uGal(R_i/F)=\G$.

For showing the existence of such an index $i$, first recognize that from 
$E=\Quot(R)$ and $\tilde{E}:=\Quot(R\otimes_F L)=\Quot(E\otimes_F L)$ one obtains
\[ \trdeg(\tilde{E}/E)=\trdeg(L/F) \quad \text{and}\quad  \trdeg(\tilde{E}/L)=\trdeg(E/F)=m\dim(\G). \]
Furthermore, any constant in $\tilde{E}$ has to be transcendental over $C$ by the second step, and $\tilde{E}^\theta$ and $L$ are algebraically independent over $C=L^\theta$. For better readability, we let $\tilde{C}:=\tilde{E}^\theta$, and 
$\tilde{C}_i:=\tilde{E}_i^\theta$.

Assume that there does not exist such an index $i$, then $\trdeg(\tilde{C}_i/C)\geq 1$ for all $i$, and hence
\[ \trdeg(\tilde{E}_i\tilde{C}/L\tilde{C})\leq \trdeg(\tilde{E}_i/L\tilde{C}_i)
\leq \trdeg(\tilde{E}_i/L) -1=\dim(\G)-1. \footnote{Composita of fields are always taken inside $\tilde{E}$.} \]

Therefore,
\[ \trdeg(\tilde{E}/L\tilde{C})\leq \sum\nolimits_{i=1}^{m}  \trdeg(\tilde{E}_i\tilde{C}/L\tilde{C})\leq m(\dim(\G)-1)=m\dim(\G)-m. \]
Furthermore,
\begin{eqnarray*} 
\trdeg(L\tilde{C}/L) &=& \trdeg(\tilde{C}/C)= \trdeg(E\tilde{C}/E) \\
&\leq & \trdeg(\tilde{E}/E)= \trdeg(L/F)=m-1, \end{eqnarray*}
and hence,
\[ \trdeg(\tilde{E}/L)=\trdeg(\tilde{E}/L\tilde{C})+ \trdeg(L\tilde{C}/L)\leq m\dim(\G)-1. \]
in contradiction to the equality $\trdeg(\tilde{E}/L)=m\dim(\G)$ from above.
%
%
\end{proof}

We are now able to proof the theorem above.

\begin{proof}[Proof of Theorem \ref{thm:realisation}]
We first assume that $\theta^{(1)}\neq 0$.
In that case by Proposition~\ref{prop:extension-of-id-by-t}, there is $t\in L$ such that $\th{1}(t)\neq 0$, and an iterative derivation $\tilde{\theta}$ on $L$ which  is equivalent to $\theta$ and which extends the iterative higher derivation $\theta_t$ by $t$ on the subfield $C(t)$.

By Theorem \ref{thm:equivalence-of-categories}, the categories of ID-modules over $(L,\theta)$ and over $(L,\tilde{\theta})$ are equivalent, and hence the affine algebraic group schemes that occur as Galois groups over  $(L,\theta)$ are the same as the ones that occur as Galois groups over  $(L,\tilde{\theta})$.

The ID-field $(L,\tilde{\theta})$ is by construction an ID-extension of $(C(t),\theta_t)$, and is finitely generated over $C(t)$ by hypothesis. By  \cite[Thm.~8.8]{bhm-mvdp:cdgt}, all reduced affine algebraic group schemes over $C$ can be realised as Galois group over $(C(t),\theta_t)$. In particular, for any such group scheme $\G$, the group schemes $\G^r$ can be realised for all $r\in \NN$.
Hence by Prop.~\ref{prop:kovacic-trick}, any reduced affine algebraic group scheme $\G$ over $C$ can be realised
over $(L,\tilde{\theta})$, and therefore also over $(L,\theta)$.

Assume now that $\theta^{(1)}=0$. As $L\neq C$, the iterative derivation $\theta$ is non-trivial, and by Lemma \ref{lem:existence-of-d}, there is $d\in \NN$ such that $\theta^{(p^d)}\neq 0$, and for all other non-zero maps $\theta^{(j)}$, the number $j$ is a multiple of $p^d$. 
The iterative derivation therefore has its image in $L\ps{T^{p^d}}$, and we can consider the map
\[  \overline{\theta}:L\to L\ps{U}, x\mapsto \sum_{j=0}^\infty \theta^{(jp^d)}(x) U^j.\]
It is not hard to check that $\overline{\theta}$ is again an iterative derivation. Furthermore, $\overline{\theta}^{(1)}\neq 0$. Hence, by the first part of the proof, any  reduced affine algebraic group scheme $\G$ over $C$ can be realised as the Galois group over $(L,\overline{\theta})$ of some Picard-Vessiot ring $(R,\overline{\theta}_R)$. Then we define $\theta_R:R\to R\ps{T}$ by 
\[  \theta_R(r):=\sum_{j=0}^\infty \overline{\theta}_R^{(j)}(r) T^{jp^d} \]
for all $r\in R$. This gives an iterative differential extension $(R,\theta_R)$ of $(L,\theta)$.
As the components $\th{k}_R$ of the iterative derivation $\theta_R$ are either zero or the same as the components of $\overline{\theta}_R$, it is immediate that $(R,\theta_R)$ even is a Picard-Vessiot ring over  $(L,\theta)$,
and that 
\[ \uGal\bigl( (R,\theta_R)/(L,\theta)\bigr) = \uGal\bigl( (R,\overline{\theta}_R)/(L,\overline{\theta})\bigr)=\G. \qedhere \]
\end{proof}

\def\cprime{$'$}

\vspace*{.5cm}

\parindent0cm

\end{document}